\documentclass{amsart}

\usepackage{amssymb,amsmath,amsthm,latexsym,amscd}
\usepackage{epsfig}
\usepackage{psfrag}
\usepackage{graphicx}

\newtheorem{Theorem}{\bf Theorem}
\newtheorem{lemma}[Theorem]{\bf Lemma}
\newtheorem{proposition}[Theorem]{\bf Proposition}
\newtheorem{corollary}[Theorem]{\bf Corollary}

\newtheorem{definition}[Theorem]{\bf Definition}

\newtheorem{theorem}[Theorem]{\bf Theorem}

\def\qed{\hfill$\Box$}
\def\scfig #1 #2 {\resizebox{#2}{!}{\includegraphics{#1}}}

\newcommand{\be}{\begin{equation}}
\newcommand{\ee}{\end{equation}}

\def\hpic #1 #2 {\mbox{$\begin{array}[c]{l} 
\epsfig{file=#1,height=#2}\end{array}$}}
\def\wpic #1 #2 {\mbox{$\begin{array}[c]{l} 
\epsfig{file=#1,width=#2}\end{array}$}}

\begin{document}

\title[Presentation of the spin planar algebra]{On a presentation of the spin planar algebra}

\author{Vijay Kodiyalam}
\address{The Institute of Mathematical Sciences, Chennai, India and Homi Bhabha National Institute, Mumbai, India}
\author{Sohan Lal Saini}
\address{The Institute of Mathematical Sciences, Chennai, India and Homi Bhabha National Institute, Mumbai, India}
\author{Sruthymurali}
\address{The Institute of Mathematical Sciences, Chennai, India and Homi Bhabha National Institute, Mumbai, India}
\author{V. S. Sunder}
\address{The Institute of Mathematical Sciences, Chennai, India and Homi Bhabha National Institute, Mumbai, India}
\email{vijay@imsc.res.in,slsaini@imsc.res.in,sruthym@imsc.res.in,sunder@imsc.res.in}
%
%
%
%
%
%
%
\begin{abstract} We define a certain abstract planar algebra by generators
and relations, study various aspects of its structure,
and then identify it with Jones' spin planar algebra.
\end{abstract}
\maketitle

%


Our goal in this note is to exhibit a presentation - a skein theory - for a very simple planar algebra,
the spin planar algebra, which is well known from the very first paper \cite{Jns1999}
of Jones on planar algebras. Our technique is to define a certain abstract planar algebra by generators
and relations, carefully study various aspects of its structure - including explicit bases for its vector spaces -
and then identify it with Jones' spin planar algebra.

We will assume throughout that the reader is familiar with planar algebras. Planar algebras are collections of vector spaces with an action by the operad of planar tangles. However, since the notion of planar algebras has been evolving since its definition in \cite{Jns1999}, to fix notations and definitions for the version
of planar algebras that we use here, we refer to \cite{DeKdy2018}. In particular, we use the version where the
vector spaces are indexed by $(k,\epsilon)$ with $k \in \{0,1,2,\cdots\}$ and $\epsilon \in \{\pm\}$ as opposed to the older version. The equivalence between these two is also shown in \cite{DeKdy2018}. Other
notions such as universal planar algebras are treated carefully in \cite{KdySnd2004} (in the context of the older planar algebras) and together these 3 papers cover all of the notions that are used here.

Let $S = \{s_1,s_2,\cdots,s_n\}$ be a finite set. We will define a planar algebra over ${\mathbb C}$ associated to this set.
Begin with the label set $L =  L_{(0,-)} = S$ equipped with the identity involution $*$. Consider the quotient $P=P(L,R)$ of the universal planar algebra $P(L)$ by the set $R$ of relations in Figures \ref{fig:modulus} and \ref{fig:mult} (where $\delta_{ij}$ denotes the Kronecker delta).

\begin{figure}[!h]
\begin{center}
\psfrag{v+}{\huge $v_+$}
\psfrag{v-}{\huge $s_i$}
\psfrag{u1}{\Huge $\displaystyle{\frac{1}{\mu(v_+)}}$}
\psfrag{u2}{\Huge $\displaystyle{\frac{1}{\mu(v_-)}}$}
\psfrag{text1}{\Huge $= \displaystyle{\sqrt{n}}$}
\psfrag{text2}{\Huge $= \displaystyle{\frac{1}{\sqrt{n}}}$}
\resizebox{12.0cm}{!}{\includegraphics{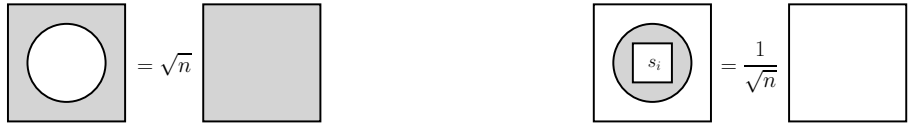}}
\end{center}
\caption{The white and black modulus relations}
\label{fig:modulus}
\end{figure}

\begin{figure}[!h]
\begin{center}
\psfrag{fi}{\Huge $s_i$}
\psfrag{fj}{\Huge $s_j$}
\psfrag{text2}{\Huge $\displaystyle{\sum\limits_i}$}
\psfrag{=}{\Huge $=\frac{1}{\sqrt{n}}$}
\psfrag{=dij}{\bf{\Huge $=\delta_{ij}$}}
\resizebox{11.0cm}{!}{\includegraphics{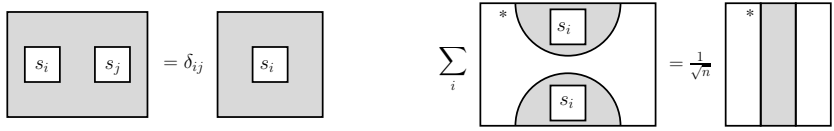}}
\end{center}
\caption{The multiplication relation and the black channel relation}
\label{fig:mult}
\end{figure}


\begin{theorem}\label{thm:spin}
The planar algebra $P$ is a finite-dimensional $C^*$-planar algebra with modulus $\sqrt{n}$ and such that $dim(P_{(0,+)}) =1$ and
$dim(P_{(0,-)})=n$. For $k>0$, $dim(P_{(k,\pm)})=n^k$ with bases as in Figures \ref{fig:bases1} and \ref{fig:bases2} for $k$ even and odd  respectively.
\begin{figure}[!h]
\begin{center}
\psfrag{si1}{\Huge $s_{i_1}$}
\psfrag{si2}{\Huge $s_{i_2}$}
\psfrag{sik}{\Huge $s_{i_m}$}
\psfrag{sj1}{\Huge $s_{j_1}$}
\psfrag{sj2}{\Huge $s_{j_2}$}
\psfrag{sjk}{\Huge $s_{j_m}$}
\psfrag{sp}{\Huge $s_{p}$}
\psfrag{sq}{\Huge $s_{q}$}
\psfrag{sim}{\Huge $s_{i_m}$}
\psfrag{sikm1}{\Huge $s_{i_{m-1}}$}
\psfrag{sjkm1}{\Huge $s_{j_{2}}$}
\psfrag{fi}{\Huge $s_i$}
\psfrag{fj}{\Huge $s_j$}
\psfrag{cdots}{\Huge $\displaystyle{\cdots}$}
\psfrag{=}{\Huge $=$}
\psfrag{=dij}{\bf{\Huge $=\delta_{ij}$}}
\resizebox{9.0cm}{!}{\includegraphics{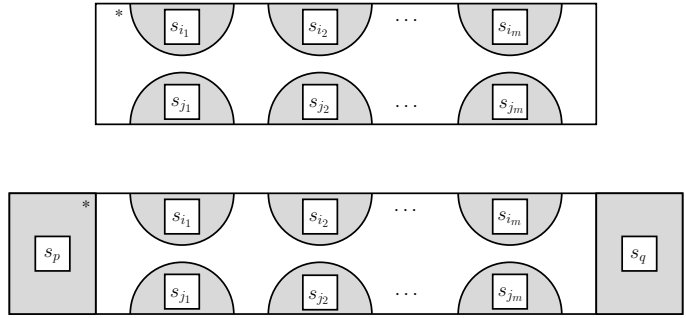}}
\end{center}
\caption{Bases ${\mathcal B}_{(2m,+)}$ for $m \geq 1$ and ${\mathcal B}_{(2m+2,-)}$ for $m \geq 0$}
\label{fig:bases1}
\end{figure}

\begin{figure}[!h]
\begin{center}
\psfrag{si1}{\Huge $s_{i_1}$}
\psfrag{si2}{\Huge $s_{i_2}$}
\psfrag{sik}{\Huge $s_{i_m}$}
\psfrag{sj1}{\Huge $s_{j_1}$}
\psfrag{sj2}{\Huge $s_{j_2}$}
\psfrag{sjk}{\Huge $s_{q}$}
\psfrag{sp}{\Huge $s_{p}$}
\psfrag{sjmm1}{\Huge $s_{j_2}$}
\psfrag{sjm}{\Huge $s_{j_m}$}
\psfrag{sikm1}{\Huge $s_{i_{m}}$}
\psfrag{simp1}{\Huge $s_{j_{m}}$}
\psfrag{sjkm1}{\Huge $s_{j_{m}}$}
\psfrag{cdots}{\Huge $\displaystyle{\cdots}$}
\psfrag{fi}{\Huge $s_i$}
\psfrag{fj}{\Huge $s_j$}
\psfrag{cdots}{\Huge $\displaystyle{\cdots}$}
\psfrag{=}{\Huge $=$}
\psfrag{=dij}{\bf{\Huge $=\delta_{ij}$}}
\resizebox{8.0cm}{!}{\includegraphics{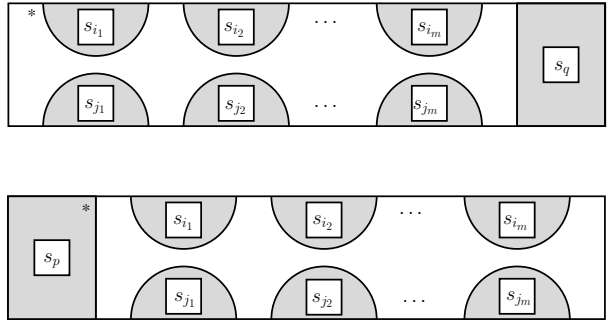}}
\end{center}
\caption{Bases ${\mathcal B}_{(2m+1,\pm)}$  for $m \geq 0$}
\label{fig:bases2}
\end{figure}
\end{theorem}

We will prove Theorem \ref{thm:spin} in small steps and put them together.

\begin{lemma}\label{lemma:unitmodulus}
The unit and modulus relations  of Figure \ref{fig:unit2} hold in the planar algebra $P$.
\begin{figure}[!h]
\begin{center}
\psfrag{v+}{\huge $v_+$}
\psfrag{v-}{\huge $s_i$}
\psfrag{u1}{\Huge $\displaystyle{\frac{1}{\mu(v_+)}}$}
\psfrag{u2}{\Huge $\displaystyle{\frac{1}{\mu(v_-)}}$}
\psfrag{xi}{\Huge $\xi,\xi$}
\psfrag{text3}{\Huge $= \displaystyle{\sqrt{n}}$}
\psfrag{text1}{\Huge $= \displaystyle{\sum\limits_{v_+ \in \mathcal{V}_+}}$}
\psfrag{text2}{\Huge $= \displaystyle{\sum\limits_{i}}$}
\resizebox{9cm}{!}{\includegraphics{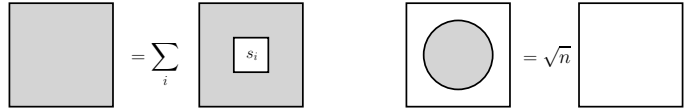}}
\end{center}
\caption{The unit and modulus relations}
\label{fig:unit2}
\end{figure}
\end{lemma}

\begin{proof}
Equivalently, what is being asserted is that the relations of Figure \ref{fig:unit2} are in the planar ideal $I(R)$
of $P(L)$ generated by $R$, which is what we will actually prove. Begin with the black channel relation, cap on the bottom and use the black modulus relations to observe that the relation of Figure \ref{fig:unit3} is in $I(R)$.

\begin{figure}[!h]
\begin{center}
\psfrag{fi}{\Huge $s_i$}
\psfrag{fj}{\Huge $s_j$}
\psfrag{text2}{\Huge $\displaystyle{\sum\limits_i}$}
\psfrag{=}{\Huge $=$}
\psfrag{=dij}{\bf{\Huge $=\delta_{ij}$}}
\resizebox{4.5cm}{!}{\includegraphics{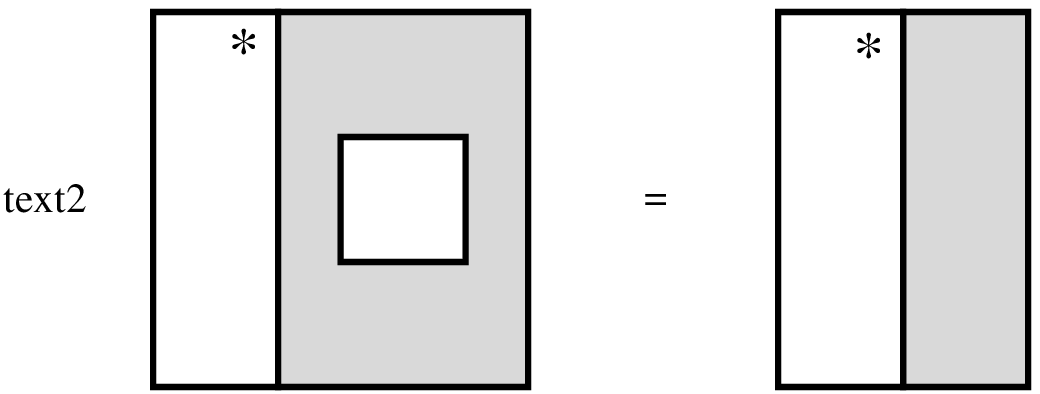}}
\end{center}
\caption{Another unit relation}
\label{fig:unit3}
\end{figure}

Both the relations of Figure \ref{fig:unit2} follow from Figure \ref{fig:unit3} - 
applying the left conditional expectation tangle and using the white modulus relation gives the relation
on the left in Figure \ref{fig:unit2}, while applying the right conditional expectation tangle and using the 
black modulus relations gives the relation
on the right in Figure \ref{fig:unit2}. 
\end{proof}

Henceforth, we will refer to both the relations on the right in Figures \ref{fig:modulus} and \ref{fig:unit2} as black modulus relations. We recall that the tangle appearing on the extreme right (respectively the extreme left) in Figure \ref{fig:unit2} is the unit tangle of colour $(0,+)$ (respectively $(0,-)$) and is denoted by
$1^{(0,+)}$ (respectively $1^{(0,-)}$).
The main step in the proof of Theorem \ref{thm:spin} is the following proposition and its corollary.

\begin{proposition}\label{prop:conn}
$dim(P_{(0,+)}) = 1$.
\end{proposition}

Before we prove Proposition \ref{prop:conn}, we will define a collection of linear functionals 
$\lambda_+: P(L)_{(0,+)} \rightarrow {\mathbb C}$
and $\lambda_{-,i}: P(L)_{(0,-)} \rightarrow {\mathbb C}$ for $i=1,2,\cdots,n$.
These are defined on bases of $P(L)_{(0,\pm)}$ as follows and extended by linearity.

A basis element of $P(L)_{(0,\pm)}$ is an $L$-labelled $(0,\pm)$-tangle, say $T$, which is just a collection of (possibly nested) closed 
loops with each of its black regions having some (possibly none) $(0,-)$ boxes labelled by elements of $L=S$. 
We will say that $T$ is inconsistently labelled if some black region of $T$ has boxes labelled by
two or more different elements of $S$ and consistently labelled otherwise.
Let $E(T)$ be the number of non-external regions of $T$ that do not have a box in it and $N(T)$ be the number of non-external regions of $T$  which have at least one labelled box (which are therefore necessarily black). Since there is 
a 1-1 correspondence between the loops of $T$ and the non-external
regions of $T$, with the loop corresponding to a region being its external boundary, $E(T) + N(T)$ is the total number of loops of $T$.

For $T$ as above of colour $(0,+)$, define $\lambda_+(T)$ as follows. 
$$
\lambda_+(T) = \left\{ \begin{array}{ll}
                        0 & {\text {if $T$ is inconsistently labelled}}\\
                        (\sqrt{n})^{E(T)-N(T)} & {\text {otherwise}}.
                        \end{array} \right.
$$
%
%
%
%
%
%
%
%
Similarly, for $T$ as above of colour $(0,-)$, define $\lambda_{-,i}(T)$ as follows. 
$$
\lambda_{-,i}(T) = \left\{ \begin{array}{ll}
                        0 & {\text {if $T$ is inconsistently labelled, or if the external }}\\
                         & {\text {region of $T$ has a box not labelled by $s_i$}}\\
                        (\sqrt{n})^{E(T)-N(T)} & {\text {otherwise}}.
                        \end{array} \right.
$$

A key `multiplicativity property' of these functionals is stated in the following lemma. By an annular $L$-labelled tangle we mean a tangle all of whose internal boxes except for one are labelled by elements of $L$. In particular, all its labelled internal boxes are coloured $(0,-)$ while the unlabelled box  may be of any colour.
The lemma below refers to the tangles $S(i)$. Here, and in the sequel, $S(i)$ will denote the $(0,-)$-tangle with a single internal $(0,-)$ box labelled $s_i$ (which appears on the right hand sides of the multiplication relation of Figure \ref{fig:mult} or of the unit relation of Figure \ref{fig:unit2}).

\begin{lemma}\label{lemma:multiplicativity} Let $A$ be an annular $L$-labelled $(0,\pm)$-tangle with its unlabelled box also of colour $(0,\pm)$. Then,\\
(a) If $A$ is of colour $(0,+)$ and its unlabelled box is also of colour $(0,+)$, and $U$ is an $L$-labelled $(0,+)$-tangle, then $\lambda_+(A \circ U) = \lambda_+(A \circ 1^{(0,+)})\lambda_+(U)$.\\
(b) If $A$ is of colour $(0,+)$ and its unlabelled box is of colour $(0,-)$, and $U$ is an $L$-labelled $(0,-)$-tangle, then $\lambda_+(A \circ U) = \sum_i \lambda_+(A \circ S(i))\lambda_{-,i}(U)$.\\
(c) If $A$ is of colour $(0,-)$ and its unlabelled box is of colour $(0,+)$, and $U$ is an $L$-labelled $(0,+)$-tangle, then for every $k$, $\lambda_{-,k}(A \circ U) = \lambda_{-,k}(A \circ 1^{(0,+)})\lambda_+(U)$.\\
(d) If $A$ is of colour $(0,-)$ and its unlabelled box is also of colour $(0,-)$, and $U$ is an $L$-labelled $(0,-)$-tangle, then for every $k$, $\lambda_{-,k}(A \circ U) = \sum_i \lambda_{-,k}(A \circ S(i))\lambda_{-,i}(U)$
\end{lemma}

\begin{proof} (a) The black regions of $A \circ U$ are of two kinds - those that correspond to the black regions of $A$ and those that correspond to the black regions of $U$. It follows that $A \circ U$ is consistently labelled iff each of $A \circ 1^{(0,+)}$ and $U$ are so.
Hence assume that all of $A \circ U$, $A \circ 1^{(0,+)}$ and $U$ are consistently
labelled. It will then suffice to see that $E(A\circ U) = E(A \circ 1^{(0,+)}) + E(U)$ and $N(A\circ U) = N(A \circ 1^{(0,+)}) + N(U)$ to finish the proof. The latter of these equations follows from the bijection between the sets of black regions, and the former since $E(T)+N(T)$ is the total number of loops of $T$, and clearly
the total number of loops of $A \circ U$ equals the sum of the total number of loops of $A \circ 1^{(0,+)}$ and of $U$ put together.\\
(b) The black regions of $A \circ U$ are of 3 kinds - those that correspond to
black regions of $A$ not containing its unlabelled box, the union of the black region of $A$  containing its unlabelled box and the external black region of $U$, and those that correspond to non-external black regions of $U$. 
Suppose that $A \circ U$ is inconsistently labelled. Thus some black region in it has boxes labelled by two or more different elements of $S$. 
If an inconsistently labelled black region of $A \circ U$ is of the first kind, then, for each $i$, $A \circ S(i)$ is inconsistently labelled. If an inconsistently labelled black region of $A \circ U$ is of the third kind, then, $U$ itself is inconsistently labelled. In the remaining case, if the black region of $A \circ U$ that contains the external region of $U$ is inconsistently labelled, then a little thought shows that for each $i$, at least one of $\lambda_+(A \circ S(i))$ or $\lambda_{-,i}(U)$ must vanish. We therefore have seen that if $\lambda_+(A \circ U)$ vanishes, then, so does $\sum_i \lambda_+(A \circ S(i))\lambda_{-,i}(U)$.

Next, suppose that $A \circ U$ is consistently labelled. We will distinguish two cases here according as the
black region of $A \circ U$ corresponding to the external region of $U$ contains a labelled box or not.
In the first case, suppose that a box labelled $s_i$ is in this black region. Then it is clear that only the $i$-term is non-zero in the RHS and so we need to check that $\lambda_+(A \circ U) = \lambda_+(A \circ S(i))\lambda_{-,i}(U)$. This will follow from checking that, in this case, $E(A\circ U) = E(A \circ S(i)) + E(U)$ and $N(A\circ U) = N(A \circ S(i)) + N(U)$. 
To prove the latter of these two equations, $N(A \circ U)$ counts the number of black regions of $A \circ U$
that have at least one labelled box. Of the three kinds of black regions of $A \circ U$ alluded to above, the first two kinds are counted in $N(A \circ S(i))$ (the second kind because of the case we're in) and the third kind in $N(U)$. The former equation follows since the sum of $N(T)$ and $E(T)$ is the number of loops of $T$, and clearly, the total number of loops of $A \circ U$ equals the sum of the total number of loops of $A \circ S(i)$ and of $U$ put together.

In the second case, the
black region of $A \circ U$ corresponding to the external region of $U$ contains no labelled boxes. Thus, neither the black region of $A$ containing its unlabelled box nor the external black region of $U$ contains a
labelled box. In this case, a little thought shows that $U$ is consistently labelled, as is $A \circ S(i)$ for
every $i$, and further that all the $\lambda_+(A \circ S(i))\lambda_{-,i}(U)$ are equal. So it suffices to see
that for some (any) $i$, $E(A \circ U) - N(A \circ U) = 2 + E(A \circ S(i)) - N(A \circ S(i)) + E(U) - N(U)$.
Here, observe that $E(A \circ U) = E(A \circ S(i)) + 1 + E(U)$ - this equations arising from considering the
three kinds of black regions of $A \circ U$. As before since the  number of loops of $A \circ U$
is the sum of the number of loops of $A \circ S(i)$ and of $U$, it follows that $N(A \circ U) = N(A \circ S(i)) - 1 + E(U)$ and so the desired equality follows.\\
(c) and (d) These follow from (a) and (b) and the observation (which follows from the definitions) that for any $L$-labelled $(0,-)$-tangle $U$, and for any $k$, $\lambda_{-,k}(U) =\sqrt{n} \lambda_+(T(k) \circ U)$, where $T(k)$ is the annular tangle in Figure \ref{fig:tk}.
\begin{figure}[!h]
\begin{center}
\psfrag{fj}{\Huge $s_k$}
\resizebox{2.5cm}{!}{\includegraphics{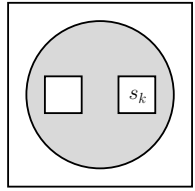}}
\end{center}
\caption{The tangle $T(k)$}
\label{fig:tk}
\end{figure}
\end{proof}

\begin{proof}[Proof of Proposition \ref{prop:conn}] 
We first show that the image of a basis element of $P(L)_{(0,+)}$ in $P_{(0,+)}$ is a multiple of $1_{(0,+)}$. Such a basis element is an $L$-labelled $(0,+)$-tangle which is just a collection of nested closed 
loops with some of its black regions having some $(0,-)$ boxes labelled by elements of $L=S$.
Using only the black and white modulus relations and working from the innermost loops outward - formally, by induction on the number of loops - it is clear that the 
image of such a basis element is a scalar multiple of $1_{(0,+)}$. Thus $dim(P_{(0,+)}) \leq 1$.

The harder part of the proof is to show non-collapse. For this, we will show that the linear functional
$\lambda_+$ defined on $P(L)_{(0,+)}$ vanishes on $I(R)_{(0,+)}$ (which is, by definition, the $(0,+)$ part of the planar ideal $I(R)$ generated by $R$) and consequently descends to $P_{(0,+)}$. Since $\lambda_+$ is
clearly surjective, this will finish the proof.

%
%
%
%
%
%
%
%

To show that $\lambda_+$ vanishes on $I(R)_{(0,+)}$, note that a spanning set of $I(R)_{(0,+)}$ consists of all
$Z_T^{P(L)}(x_1 \otimes x_2 \otimes \cdots \otimes x_b)$ where $T$ is a $(0,+)$-tangle with internal boxes of colour $(k_j,\epsilon_j)$ for $j = 1,2,\cdots,b$, $x_i \in R$ for one $i$ and all other $x_j$ are basis elements of $P(L)_{(k_j,\epsilon_j)}$. Since $R$ is non-empty only in colours $(0,\pm)$ and $(2,+)$, $(k_i,\epsilon_i)$ is
necessarily one of these colours.

Consider the annular $L$-labelled $(0,+)$-tangle, say $A$, obtained from $T$ by substituting the basis elements $x_j$ in all the boxes of $T$ except for the $i^{th}$-box. What we need to see is that for each
of the relations in $R$, substituting the left hand side of the relation and evaluating $\lambda_+$ on the
resulting (not necessarily basis) element of $P(L)_{(0,+)}$ so obtained gives the same result as substituting the right hand side of the relation and evaluating $\lambda_+$ on the resulting element of $P(L)_{(0,+)}$.

This easily follows from Lemma \ref{lemma:multiplicativity} for each of the white and
black modulus relations in Figure \ref{fig:modulus} as well as for the multiplication relation in Figure \ref{fig:mult}. For instance, we show how this works for the multiplication relation. What we need to see here is that for every annular $(0,+)$-tangle $A$ with unlabelled box of colour $(0,-)$, 
$$
\lambda_+(A \circ W(i,j)) = \delta_{ij} \lambda_+(A \circ S(i)),
$$
where $W(i,j)$ is the $L$-labelled $(0,-)$-tangle that appears on the left side of the multiplication relation of Figure \ref{fig:mult}. By multiplicativity, this is equivalent to verifying that
$$
\sum_k \lambda_+(A \circ S(k))\lambda_{-,k}(W(i,j)) = \delta_{ij} \sum_k \lambda_+(A \circ S(k))\lambda_{-,k}(S(i)).
$$
We immediately reduce to checking that  $\lambda_{-,k}(W(i,j)) = \delta_{ij}\lambda_{-,k}(S(i))$ for every $k$, which is indeed true by definition of the functional $\lambda_{-,k}$.

Invariance under the black channel relation needs a little work.
Consider the structure of the annular tangle $A$ in this case. It is a $(0,+)$-tangle with  a single unlabelled $(2,+)$-box and (possibly) a number of labelled $(0,-)$-boxes. Since the strings impinging on the $(2,+)$-box have to be connected among themselves, there are six (classes of) possibilities for the region around the $(2,+)$-box - as shown in Figure \ref{fig:Ttilde}.
\begin{figure}[!h]
\begin{center}
\psfrag{fi}{\Huge $s_i$}
\psfrag{fj}{\Huge $s_j$}
\psfrag{text2}{\Huge $\displaystyle{\sum\limits_i}$}
\psfrag{=}{\Huge $=$}
\psfrag{=dij}{\bf{\Huge $=\delta_{ij}$}}
\resizebox{9cm}{!}{\includegraphics{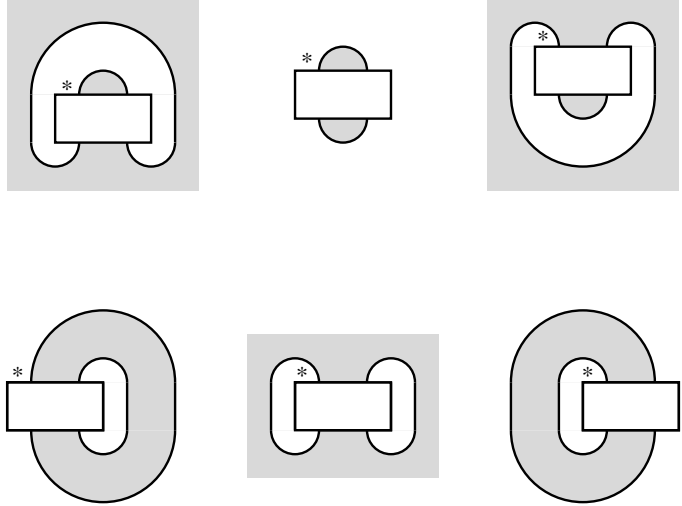}}
\end{center}
\caption{Region around the unlabelled $(2,+)$-box of $A$}
\label{fig:Ttilde}
\end{figure}
Note that in each of these figures, we have not shown further detail within each of the (two) bounded regions
in each case. Explicitly, for instance, the top left picture in Figure \ref{fig:Ttilde} actually stands for a picture
of the type shown in Figure \ref{fig:detail},
\begin{figure}[!h]
\begin{center}
\psfrag{fi}{\Huge $s_i$}
\psfrag{fj}{\Huge $s_j$}
\psfrag{text2}{\Huge $\displaystyle{\sum\limits_i}$}
\psfrag{=}{\Huge $=$}
\psfrag{=dij}{\bf{\Huge $=\delta_{ij}$}}
\resizebox{4cm}{!}{\includegraphics{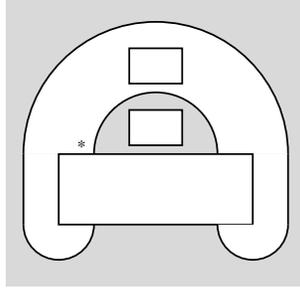}}
\end{center}
\caption{Detail of the top left figure in Figure \ref{fig:Ttilde}}
\label{fig:detail}
\end{figure}
where the two unlabelled boxes in the bounded regions contain some labeled $(0,\pm)$-tangles according
to their colour.

Again, using Lemma \ref{lemma:multiplicativity} as above, we reduce to showing that when $A$ is one of the 6 annular tangles in Figure \ref{fig:Ttilde}, substituting the left hand side of the black channel relation into the
$(2,+)$-box of $A$ and evaluating $\lambda_+$ or $\lambda_{-,k}$ on the result (according as $A$ is
of colour $(0,+)$ or $(0,-)$) gives the same result as doing the same for the right hand side.

This involves a series of checks. We do one of them and leave the rest (which are all similar) to the  conscientious reader. Suppose, for instance, that $A$ is the annular $(0,-)$-tangle in the middle of the
bottom row of Figure \ref{fig:Ttilde}. What needs to be checked in this case is that for every $k$, the equation of Figure \ref{fig:check} holds.
\begin{figure}[!h]
\begin{center}
\psfrag{fi}{\Huge $s_i$}
\psfrag{fj}{\Huge $s_j$}
\psfrag{text2}{\Huge $\displaystyle{\sum\limits_i} \lambda_{-,k}$}
\psfrag{=}{\Huge $=\frac{1}{\sqrt{n}} \lambda_{-,k}$}
\psfrag{=dij}{\bf{\Huge $=\delta_{ij}$}}
\resizebox{10cm}{!}{\includegraphics{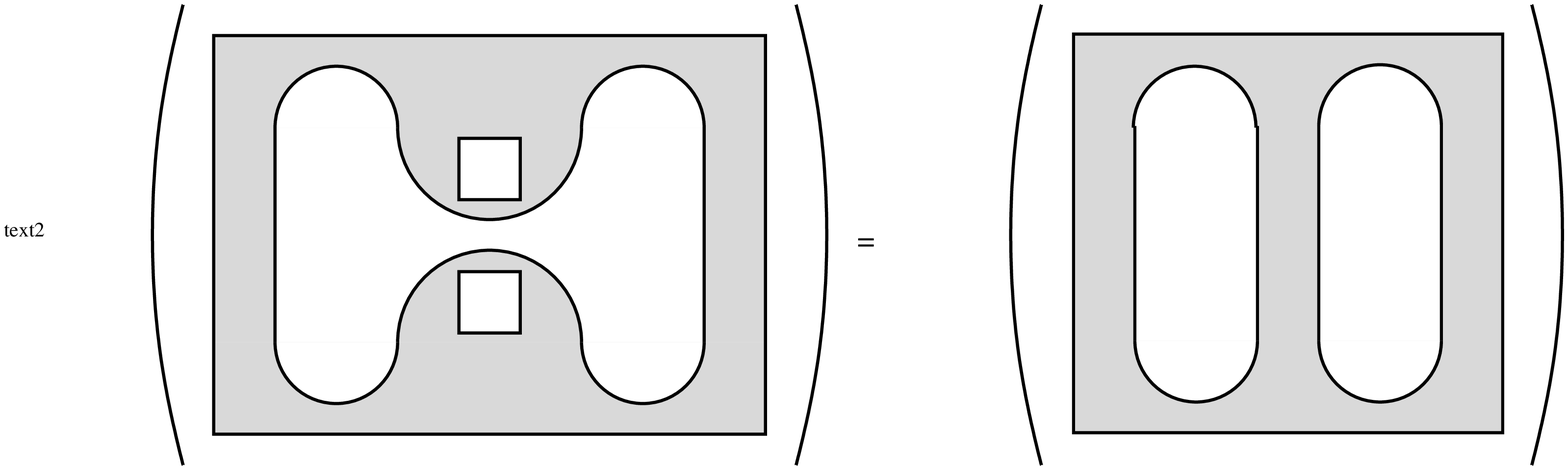}}
\end{center}
\caption{Equation to be checked}
\label{fig:check}
\end{figure}
Now, by definition of $\lambda_{-,k}$, the only term that survives on the left is when $i=k$ and this evaluates
to $\sqrt{n}$ which is easily seen to be exactly what the right side also evaluates to.
\end{proof}

\begin{corollary}\label{cor:conn}
$dim(P_{(0,-)}) = n$.
\end{corollary}

\begin{proof}
As in the proof of Proposition \ref{prop:conn}, by induction on the number of loops, it is clear that the image of a basis element of $P(L)_{(0,-)}$ in $P_{(0,-)}$ is either a multiple of $S(i)$ or of $1_{(0,-)}$ (according as its external region has a box labelled $s_i$ or is empty). Now, using Lemma \ref{lemma:unitmodulus}, it is clear that the images of the $S(i)$ in $P_{(0,-)}$ are a spanning set.

To show linear independence, suppose that $\sum_i \alpha_i S(i)$ is in $I(R)_{(0,-)}$. Applying the annular tangle $T(k)$ of Figure \ref{fig:tk} to this must therefore yield an element of $I(R)_{(0,+)}$. Hence a further application of $\lambda_+$ should give $0$ by Proposition \ref{prop:conn}. However, a direct calculation using the definition of $\lambda_+$
gives $\frac{1}{\sqrt{n}}\alpha_k$. Thus, all the $\alpha_i$ must vanish.
\end{proof}

\begin{corollary}
The planar algebra $P$ has modulus $\sqrt{n}$.
\end{corollary}

\begin{proof}
The planar algebra $P$ does not collapse by Proposition \ref{prop:conn} and the white modulus relation of Figure \ref{fig:modulus} and the black modulus relation of Figure \ref{fig:unit2} show that $P$ has modulus
$\sqrt{n}$.
\end{proof}

To show that $P$ is a $*$-planar algebra, we will appeal to the following general simple lemma whose proof we omit.

\begin{lemma} Let $P = P(L,R)$ for some label set $L$ and some set of relations $R$ in $P(L)$.
Suppose that $L$ is equipped with an involution $*$ (by which we mean that each $L_{k,\epsilon)}$ is) such that for every relation in $R$, its adjoint
is also in the planar ideal generated by $R$. Then $P$ has a natural $*$-planar algebra structure.\qed
\end{lemma}

\begin{corollary}\label{cor:*pl}
$P$ is a $*$-planar algebra.
\end{corollary}

\begin{proof}
 Inspection  shows that each of the relations in $R$ is, in fact, invariant under the involution on $L$  given
 by the identity map.
\end{proof}

Before stating the next lemma, we recall that a Temperley-Lieb tangle is one that has no internal boxes while a pure Temperley-Lieb tangle is one which, in addition, has no closed strings. It is well known - see Proposition 2.8.1 of \cite{GdmHrpJns1989} that each pure Temperley-Lieb tangle is a monomial in the Jones projections.

\begin{lemma}\label{lemma:span}
The elements of $P_{(2m,\pm)}$ and $P_{(2m+1,\pm)}$ shown in Figures \ref{fig:bases1} and \ref{fig:bases2}
are spanning sets.
\end{lemma}

\begin{proof}
It suffices to see that ${\mathcal B}_{(2m,+)}$ and ${\mathcal B}_{(2m+1,+)}$ - the pictures on top in Figures \ref{fig:bases1} and \ref{fig:bases2} - are spanning sets
of $P_{(2m,+)}$ and $P_{(2m+1,+)}$ respectively since ${\mathcal B}_{(2m,-)}$ and ${\mathcal B}_{(2m+1,-)}$ are obtained
by a rotation from these, and rotation implements an isomorphism from $P_{(k,+)}$ to $P_{(k,-)}$ for $k > 0$.

To show that ${\mathcal B}_{(2m,+)}$ spans $P_{(2m,+)}$, begin with an
arbitrary basis element of $P(L)_{(2m,+)}$. This is an $L$-labelled $(2m,+)$-tangle. Using the modulus relations,
we may assume that this tangle has no closed loops in it. Thus it is a Temperley-Lieb tangle with some of its black regions having some labelled $(0,-)$-boxes. Using the unit relation of Figure \ref{fig:unit3} together with the multiplication relation, we may express this as a linear combination of Temperley-Lieb tangles each of whose black regions has a single labelled $(0,-)$-box and it now suffices to see that any such tangle is in the 
span of ${\mathcal B}_{(2m,+)}$.

Observe now that any Temperley-Lieb tangle each of whose black regions has a single labelled $(0,-)$-box
 can  be expressed as a product of three tangles, the middle one of which is a pure Temperley-Lieb tangle and the other two being tangles as in Figure \ref{fig:topbottom}.
 \begin{figure}[!h]
\begin{center}
\psfrag{si1}{\Huge $s_{i_1}$}
\psfrag{s12}{\Huge $s_{i_2}$}
\psfrag{sim}{\Huge $s_{i_m}$}
\psfrag{sj1}{\Huge $s_{j_1}$}
\psfrag{sj2}{\Huge $s_{j_2}$}
\psfrag{sjk}{\Huge $s_{i_{m+1}}$}
\psfrag{simm1}{\Huge $s_{i_{3}}$}
\psfrag{simp1}{\Huge $s_{j_{m}}$}
\psfrag{sjkm1}{\Huge $s_{j_{m}}$}
\psfrag{cdots}{\Huge $\displaystyle{\cdots}$}
\psfrag{fi}{\Huge $s_i$}
\psfrag{fj}{\Huge $s_j$}
\psfrag{cdots}{\Huge $\displaystyle{\cdots}$}
\psfrag{=}{\Huge $=$}
\psfrag{=dij}{\bf{\Huge $=\delta_{ij}$}}
\resizebox{8.0cm}{!}{\includegraphics{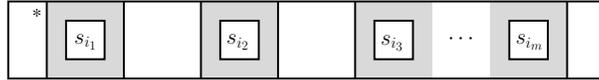}}
\end{center}
\caption{Form of the top and bottom tangles}
\label{fig:topbottom}
\end{figure}
A little thought now shows that it suffices to verify that a pure Temperley-Lieb tangle is in the span of ${\mathcal B}_{(2m,+)}$.

It is clear that the span of ${\mathcal B}_{(2m,+)}$ is closed under multiplication and so we reduce to showing that the Jones projections
are in the span of ${\mathcal B}_{(2m,+)}$.  Using the black channel relation and the unit relation, this fact is obvious for the even Jones
projections. A little calculation shows that this is also true for the odd Jones projections again using the black channel relation ``horizontally"
along with more applications of the unit and black channel relations.

A similar proof works to show that ${\mathcal B}_{(2m+1,+)}$ spans $P_{(2m+1,+)}$.
\end{proof}

Before we state the next lemma, we will introduce the following notation for certain elements in $P$. We denote $(\sqrt{n})^{m}$ times the elements on the top and bottom in Figure \ref{fig:bases1} by $e^{i_1\cdots i_m}_{j_1\cdots j_m}$ and $e[p)^{i_1 \cdots i_{m}}_{j_1 \cdots j_m}(q]$ respectively. Similarly, we denote $(\sqrt{n})^{m}$ times the
elements on the top and bottom in Figure \ref{fig:bases2} by $e^{i_1\cdots i_m}_{j_1\cdots j_m}(q]$ and $e[p)^{i_1 \cdots i_{m}}_{j_1 \cdots j_m}$. We omit the proof of the following lemma which, among other things, justifies these notations.

\begin{lemma}\label{lemma:notation}
The following relations hold in $P$.
\begin{eqnarray*}
e^{i_1\cdots i_m}_{j_1\cdots j_m}.e^{k_1\cdots k_m}_{l_1\cdots l_m}&=& \delta_{j_1k_1}\cdots \delta_{j_mk_m} e^{i_1\cdots i_m}_{l_1\cdots l_m}\\
e[p)^{i_1 \cdots i_{m}}_{j_1 \cdots j_m}(q].e[r)^{k_1 \cdots k_{m}}_{l_1 \cdots l_m}(s]&=& \delta_{pr} \delta_{j_1k_1} \cdots \delta_{j_mk_{m}}\delta_{qs} e[p)^{i_1 \cdots i_{m}}_{l_1 \cdots l_m}(s]\\
e^{i_1\cdots i_m}_{j_1\cdots j_m}(q].e^{k_1\cdots k_m}_{l_1\cdots l_m}(s] &=& 
\delta_{j_1k_1}\cdots \delta_{j_mk_m}\delta_{qs} e^{i_1\cdots i_m}_{l_1\cdots l_m}(q]\\
e[p)^{i_1 \cdots i_{m}}_{j_1 \cdots j_m}.e[r)^{k_1 \cdots k_{m}}_{l_1 \cdots l_m} &=& \delta_{pr}\delta_{j_1k_1} \cdots \delta_{j_mk_{m}}e[p)^{i_1 \cdots i_{m}}_{l_1 \cdots l_m}.
\end{eqnarray*}
\end{lemma}




To see that ${\mathcal B}_{(2m,\pm)}$ and ${\mathcal B}_{(2m+1,\pm)}$ are bases of 
 $P_{(2m,\pm)}$ and $P_{(2m+1,\pm)}$ respectively, we consider the natural traces on $P_{(k,\pm)}$ defined using the
 picture trace tangles and normalised appropriately. Explicitly, define the normalised trace $\tau$ on $P_{(k,+)}$ by $\tau(x)1^{(0,+)} = (\sqrt{n})^{-k}Z_{TR^{(0,+)}}(x)$.
Since $P_{(0,-)}$ is $n$-dimensional with basis $\{S(i): i=1,2\cdots,n\}$, define $\tau$ on $P_{(k,-)}$ by $\tau(x) = (\sqrt{n})^{-k}tr(Z_{TR^{(0,-)}}(x))$, where $tr: P_{(0,-)} \rightarrow {\mathbb C}$ is defined by $tr(S(i)) = n^{-1}$ for every $i=1,\cdots,n$.
Simple calculation using the relations shows that the normalised trace of each of the basis elements $e^{i_1\cdots i_m}_{j_1\cdots j_m}$, $e[p)^{i_1 \cdots i_{m}}_{j_1 \cdots j_m}(q]$, $e^{i_1\cdots i_m}_{j_1\cdots j_m}(q]$ and $e[p)^{i_1 \cdots i_{m}}_{j_1 \cdots j_m}$ is given by $\delta_{i_1j_1}\cdots \delta_{i_mj_m}$ times $n^{-m},n^{-m-2},n^{-m-1}$ and $n^{-m-1}$ respectively.

The following corollary is an immediate consequence of Lemma \ref{lemma:span}  and Lemma \ref{lemma:notation}.
\begin{corollary}\label{cor:onb}
${\mathcal B}_{(k,\pm)}$ 
are orthogonal bases of 
 $P_{(k,\pm)}$. For each basis element $e$, $\tau(e^*e) > 0$.
\end{corollary}

\begin{proof}
First, Lemma \ref{lemma:notation} and the normalised trace computation show that for each basis element $e$, $\tau(e^*e) > 0$.
Next, by Lemma \ref{lemma:span}, ${\mathcal B}_{(k,\pm)}$ are spanning sets and again by Lemma \ref{lemma:notation} and the normalised trace computation, they are orthogonal and hence linearly independent. 
\end{proof}

Before we prove Theorem \ref{thm:spin} we define what we mean by a $C^*$-planar algebra since this definition does not explicitly appear in the references alluded to earlier.

\begin{definition}
A $*$-planar algebra $P$ is said to be a $C^*$-planar algebra if there exist positive normalised traces
$\tau_{\pm}: P_{(0,\pm)} \rightarrow {\mathbb C}$ such that all the traces $\tau_{\pm} \circ Z_{TR^{(0,\pm)}}$ defined on $P_{(k,\pm)}$ are faithful and positive.
\end{definition}

\begin{proof}[Proof of Theorem \ref{thm:spin}]
Most of the proof of the theorem is contained in Proposition \ref{prop:conn} and Corollary \ref{cor:conn} (for the dimensions of $P_{0,\pm)}$), Corollary \ref{cor:*pl} (for $P$ being a $*$-planar algebra) and Corollary \ref{cor:onb} (for the bases of $P_{(k,\pm)}$). What remains to be seen is that
$P$ is a $C^*$-planar algebra, or equivalently, that the normalised traces defined above are faithful positive traces. This fact also follows easily from Corollary \ref{cor:onb}.
\end{proof}

In our final result we will identify the planar algebra $P$ with the planar algebra of the bipartite graph with one even vertex and $n$ odd vertices, which in turn can be identified with the spin planar algebra. We merely sketch the proof of the identification leaving out most details. These details are routine computations to verify that certain relations hold in the planar algebra of the bipartite graph. A template for such a proof
appears in Proposition 2  of \cite{DeKdy2018}.

Recall that the spin planar algebra is defined in Example 2.8 of \cite{Jns1999} by explicitly specifying its vector spaces and then defining the action of planar tangles using  ``maxima and minima" of strings in the tangle. We will not need the detailed definition in this paper. However what we will need is the identification
of this planar algebra with that of a planar algebra associated to a specific bipartite graph - see Example 4.2
of \cite{Jns2000}.
Recall that for a finite connected bipartite graph $\Gamma$ with vertex set $V = V_+ \coprod V_-$ and edge set $E$, the planar algebra $P(\Gamma)$ of the bipartite graph has vector spaces
given by $P(\Gamma)_{(k,\pm)}$ being the vector space with basis all loops of length $2k$ in $\Gamma$ based at a vertex in $V_\pm$. The description of the action of tangles on these vector spaces requires a choice of spin function $V \rightarrow {\mathbb R}_+$ which we will take to be the co-ordinate-wise square root of a Perron-Frobenius eigenvector for the graph - appropriately normalised. The details of the construction of $P(\Gamma)$ are set out in \cite{Jns2000}. 

\begin{proposition}
Let $\Gamma$ be the bipartite graph in Figure \ref{fig:bipartite} below.
\begin{figure}[!h]
\begin{center}
\psfrag{v1}{\huge $v_1$}
\psfrag{v2}{\huge $v_2$}
\psfrag{vnm1}{\huge $v_{n-1}$}
\psfrag{vn}{\huge $v_{n}$}
\psfrag{w}{\huge $w$}
\psfrag{v}{\Huge $\vdots$}
\resizebox{3.0cm}{!}{\includegraphics{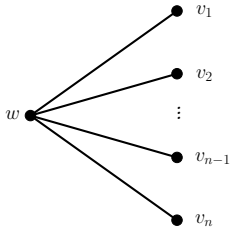}}
\end{center}
\caption{The bipartite graph $\Gamma$}
\label{fig:bipartite}
\end{figure}

\noindent
With $S = \{v_1,\cdots,v_n\}$, the planar algebra $P$
of Theorem \ref{thm:spin} is isomorphic to $P(\Gamma)$ by the map that takes $v_i \in P_{(0,-)}$ to the loop of length 0
based at $v_i$ in $P(\Gamma)_{(0,-)}$.
\end{proposition}

\begin{proof}[Sketch of proof] There is clearly a map of planar algebras from the universal planar algebra on the label set $L = L_{(0,-)} = \{v_1,\cdots,v_n\}$ to $P(\Gamma)$ defined by the above prescription. This map is surjective since it is easy to see
that the planar algebra $P(\Gamma)$ is generated by $P(\Gamma)_{(0,-)}$.
Using the explicit description of the action
of tangles in $P(\Gamma)$, this map is verified to commute with the action of all generating tangles, thus descending to the quotient $P$. Finally, observing  that the dimensions of
the $(k,\pm)$ spaces on both sides are given by $n^{k}$ (except that the $(0,+)$ spaces have dimension $1$) concludes the proof.
\end{proof}

\section*{acknowledgements}
The third named author gratefully acknowledges the partial support of the Swarnajayanti fellowship grant No. 
DST/SJF/MSA-02/2014-15 of Prof. Amritanshu Prasad.

\end{document}